\documentclass[12pt, a4paper]{amsart}
\usepackage[centering]{geometry}
\usepackage{amssymb}
\usepackage{mathrsfs}% for \mathscr
\usepackage{hyperref}
\usepackage{graphicx}
\usepackage{enumerate}
\usepackage[all,cmtip]{xy}
\usepackage{amscd}

\numberwithin{equation}{section}
\usepackage{amsthm}
\newtheorem{thm}{Theorem}[section]

\newtheorem{lemma}{Lemma}[section]
\newtheorem{conj}{Conjecture}[section]
\newtheorem{defn}{Definition}[section]

\theoremstyle{definition}

\newcommand{\diam}{\operatorname{diam}}
\newcommand{\dist}{\operatorname{dist}}

\title{Koebe uniformization of nondegenerate domains with bounded gap-ratio} 
\author{Yi Zhong}
\address{Institute of Mathematical Sciences and Applications, NingboTech University, Ningbo, 315100, China}
\email{yizhong@zju.edu.cn}
\makeatletter
\@namedef{subjclassname@2020}{%
  \textup{2020} Mathematics Subject Classification}
\makeatother


\begin{document}

\begin{abstract}
Koebe uniformization is a fundemental problem in complex analysis. In this paper, we use transboundary extremal length to show that every nondegenerate and uncountably connected domain with bounded gap-ratio is conformally homeomorphic to a circle domain.
\end{abstract}

\maketitle 

\setcounter{tocdepth}{1}
% \addtocontents{toc}{\protect\setcounter{tocdepth}{1}}
\tableofcontents

\section{Introduction}\label{sec1}
A connected domain $\Omega$ on the Riemannian sphere $\hat{\mathbb{C}}$ is called a circle domain if each component of $\hat{\mathbb{C}}\setminus\Omega$ is either a closed disk or a point. In 1909, P. Koebe \cite{Koebe1} posed the following conjecture, known as the Kreisnormierungsproblem:
\begin{conj}[Koebe's conjecture]
	Any plane domain is conformally homeomorphic to a circle domain in $\hat{\mathbb{C}}$.
\end{conj}
In the 1920's, P. Koebe \cite{Koebe2} himself proved the following theorem, which was able to confirm the
conjecture in the finitely connected case:
\begin{thm}\label{multi-Koebe}
	Let $\Omega\subset\hat{\mathbb{C}}$ be a finite multi-connected domain. Then $\Omega$ is conformally homeomorphic to a circle domain.
\end{thm}
Later, P. Koebe \cite{Koebe3} showed that the conjecture is true for a class of domains with some symmetry. For domains with various conditions on the ``limit boundary components'', one can see \cite{Bers,Denneberg,Grotzsch,Haas,Meschowski1,Meschowski2,Sario,Sterbel1,Sterbel2}.

In 1993, Z.X. He and O. Schramm \cite{H-S1} proved that the Koebe's conjecture holds for countably connected domains, which is a major breakthrough. Recently, K. Rajala \cite{Rajala} presented a new proof for this case by the idea of exhaustion.
\vskip0.15in
For uncountablly connected cases, O. Schramm \cite{Schramm} introduced the tool of transboundary extremal length, which has played a central role in recent developments on the  uniformization of fractal metric spaces, see \cite{Bonk1,B-M1,B-M2,H-L}. 
Then he prescribed the boundary shapes by studying the so called ``cofat domains'':
\begin{defn}\label{fat-cofat}
	Let $\tau>0$ be some constant. A set $A\subset\hat{\mathbb{C}}$ will be called $\tau-$fat, if for every $x\in A\cap\mathbb{C}$ and for every disk $B=B(x,r)$ centered at $x$ that does not contain $A$ we have $area(A\cap B)\geq\tau\cdot area(B)$. A connected domain $\Omega\subset\hat{\mathbb{C}}$ is cofat, if each connected component of its complement is $\tau-$fat for some $\tau>0$.
\end{defn}
By the tool of transboundary extremal length, Schramm established the following.
\begin{thm}\label{cofat}
	Every cofat domain in $\hat{\mathbb{C}}$ is conformally homeomorphic to a circle domain.
\end{thm}

Resently X.G. Wang and Y. Zhong demonstrate that any infinitely connected attracting Fatou domain of a geometrically finite rational map is conformally homeomorphic to a cofat domain  \cite{W-Z}. More resent results related to the Koebe uniformization problem can be found in \cite{Bonk2,H-M,H-S2,H-S3,N-Y,Younsi}. However, the conjecture is still open.

\subsection{Main results}
Since the boundary components can be quite complicated, domains do not usually satisfy the ``cofat'' condition.
It is natural to study a class of more general domains. We will mention some necessary definitions herein to make the presentation clear and accessible.
\begin{defn}
	Let $\kappa>0$ be some constant. A connected set $E\subset\hat{\mathbb{C}}$ with positive area is called $\kappa-$nondegenerate if 
	$$area(E)\geq\kappa\cdot\diam(E)^2.$$
	A connected domain $\Omega\subset\hat{\mathbb{C}}$ is called nondegenerate if all its complementary components are $\kappa-$nondegenerate for some $\kappa>0$.
\end{defn}
Let $E$ be a nondegenerate connected subset in $\hat{\mathbb{C}}$. We denote
$$\kappa(E)=\frac{area(E)}{\diam(E)^2}.$$
Clearly, $\kappa(E)=\pi/4$ when $E$ is a closed disk; otherwise we must have $0<\kappa(E)<\pi/4$. Recall definition \ref{fat-cofat}, it is easy to show that a cofat domain is also a nondegenerate domain.

\vskip0.15in
The method used in establishing Koebe uniformization of cofat domains can not be applied directly to nondegenerate domains. The reason for this limitation lies in the more complicated boundary behavior exhibited by the complementary of nondegenerate domains. In section \ref{sec3}, we will introduce a new quentity named gap-ratio to describe the distribution of those complementary components. By using transboundary extremal length and exploring the geometrical properties of such domains, we show that the  bounded gap-ratio property will overcome the obstacle induced by its irregular boundary. Consequently, we obtain the main theorem of this paper:

\begin{thm}\label{nondegenerate}
	Every nondegenerate domain in $\hat{\mathbb{C}}$ with bounded gap-ratio is conformally homeomorphic to a circle domain.
\end{thm}

\subsection{Outline of the paper}
This paper is organized as follows. In section \ref{sec2} we provide some preliminaries. Section \ref{sec3} is devoted to prove that a domain with bounded gap-ratio has well-distributed property, and we established the extended Carath\'eodory kernel convergence theorem for this case in section \ref{sec4}. The remaining part of the proof, arranged in section \ref{sec5}, will proceed naturally.
\vskip0.15in
The author would like to express his gratitude to Xiaoguang Wang for many useful comments during the formative period of this work in Zhejiang University.

\section{Preliminaries}\label{sec2}
Let $\Omega\subset\hat{\mathbb{C}}$ be a multi-connected domain and $\mathcal{E}(\Omega)=\hat{\mathbb{C}}/\sim$, where $z_1\sim z_2$ if and only if $z_1$ and $z_2$ belong to the same connected component of $\hat{\mathbb{C}}-\Omega$. The space $\mathcal{E}(\Omega)$ is the ends compactification of $\Omega$ and we define $\mathcal{C}(\Omega)=\mathcal{E}(\Omega)-\Omega$ as the complementary space of $\Omega$. Let $\pi_{\Omega}:\hat{\mathbb{C}}\rightarrow\mathcal{E}(\Omega)$ be the quotient map. The notation $\ulcorner p\urcorner$ will stand for $\pi^{-1}_{\Omega}(p)$ when $p\in\mathcal{E}(\Omega)$ or $p\subset\mathcal{E}(\Omega)$. We say $p\in\mathcal{C}(\Omega)$ is a non-trivial complementary component if its diameter $\diam(\ulcorner p\urcorner)>0$. Otherwise, $p$ is a trivial complementary component. The space of trivial and non-trivial complementary components will be denoted by $\mathcal{C}_t(\Omega)$ and $\mathcal{C}_{nt}(\Omega)$. 

\subsection{Transboundary extremal length}
Let the area measure $\sigma$ on the space $\mathcal{E}(\Omega)$ be equal to Lebesgue measure on $\Omega$ 
and equal to counting measure on $\mathcal{C}(\Omega)$. 
An extended metric for the domain $\Omega$ is a Borel measurable function $m:\mathcal{E}(\Omega)\rightarrow[0,\infty)$. 
The $m-$area of $\mathcal{E}(\Omega)$ is
$$A(m)=\int_{\mathcal{E}(\Omega)}m^2\mathrm{d}\sigma=\parallel m\parallel^2_2.$$
The extended metric $m$ is allowable if $A(m)<\infty$.
Let $I\subset\mathbb{R}$ be an interval, define $\gamma:I\rightarrow\mathcal{E}(\Omega)$ be a curve in $\mathcal{E}(\Omega)$, then the $m-$length of $\gamma$ is
$$L_{m}(\gamma)=\int_{\gamma^{-1}(\Omega)}m(\gamma(t))\mid\mathrm{d}\gamma(t)\mid
+\sum_{p\in\mathcal{C}(\Omega)\cap\gamma(I)}m(p).$$
Let $\Gamma$ be a collection of curves in $\mathcal{E}(\Omega)$. For any extended metric $m$, we set 
$$L_m(\Gamma)=\inf_{\gamma\in\Gamma}L_m(\gamma)$$
be the $m-$length of $\Gamma$.
Then the transboundary extremal length of $\Gamma$ is defined as
$$EL(\Gamma)=\sup_{m}\frac{L_m(\Gamma)^2}{A(m)},$$
where the supremum is taken over all allowable extended metrics on $\mathcal{E}(\Omega)$.
\begin{lemma}[Invariance of transboundary extremal length (\cite{Schramm})]
	Let $f:\Omega\rightarrow\Omega^*$ be a conformal homeomorphism between domains in $\hat{\mathbb{C}}$, and let $\Gamma$ be a collection of curves in $\mathcal{E}(\Omega)$. Set $\Gamma^*=\{f\circ\gamma:\gamma\in\Gamma\}$. 
	Then $EL(\Gamma)=EL(\Gamma^*)$.
\end{lemma}

\subsection{Well-distributed property}
Clearly, the transboundary extremal length is a generialization of traditional extremal length. 
We want to apply it to consider the distance betwenn two subsets in a multi-connected domain.

For each $b\in\mathcal{C}(\Omega)$ and $q\in\mathcal{E}(\Omega)$, let $\beta(b)$ be a Jordan curve in $\Omega-\{\infty\}$ that separates $b$ from $\infty$ in $\mathcal{E}(\Omega)$, $N(\Omega,b)$ be the connected component of $\mathcal{E}(\Omega)-\beta(b)$ that contains $b$ and $\Gamma_{\Omega}(q,b)$ be all the Jordan curves in $N(\Omega,b)-\{q,b\}$ satisfying: 
(i) $\gamma$ separate $\{q,b\}$ from $\infty$; 
or (ii) $\bar{\gamma}=\gamma\cup\{b\}$ and $\bar{\gamma}$ separate $q$ from $\infty$. 
Let $EL_{\Omega}(q,b)$ be the transboundary extremal length of $\Gamma_{\Omega}(q,b)$.
We introduce the following definition:
\begin{defn}[Well-distributed]\label{well-distributed}
	Let $\Omega\subset\hat{\mathbb{C}}$ be a multi-connected domain with $\infty\in\Omega$. 
	For each $b\in\mathcal{C}(\Omega)$ and for any $\varepsilon>0$, 
	there is some $\delta>0$ such that for all $q\in\mathcal{E}(\Omega)$ and all $\Omega_0$, 
	which is a union of $\Omega$ and some components of $\hat{\mathbb{C}}-\Omega$,
	the transboundary extremal length $EL_{\Omega_0}(q,b)<\varepsilon$ 
	when the distance between $\ulcorner q\urcorner$ and $\ulcorner b\urcorner$ is less than $\delta$. 
	Then we say $\mathcal{C}(\Omega)$ is well-distributed or $\Omega$ has well-distributed property.
\end{defn}

We note that the proof of theorem \ref{nondegenerate} relies on the key observation that 
under what condition a nondegenerate domain has well-distributed property. 
In the following section, we will present the notion of gap-ratio and prove that 
$\mathcal{C}(\Omega)$ is well-distributed when $\Omega$ is a nondegenerate domain with bounded gap-ratio.

\section{Bounded gap-ratio implies well-distributed property}\label{sec3}
Let us consider a connect compact set $K\subset\mathbb{C}$ as shown in figure \ref{gapratio}. 
Given some point $w\in\mathbb{C}$. If $w\not\in K$, 
we can use the distance $\dist(w,K)$ to discribe the relevant location bewteen them. 
Here we introduce a new quantity to measure the relationship 
between the shape of $K$ and the distance from $K$ to $w$. 
Let
\begin{equation*}
	Gr(K,w)=\frac{\displaystyle\sup_{z\in K}\mid z-w\mid}{\displaystyle\inf_{z\in K}\mid z-w\mid}.
\end{equation*}
We define $Gr(K,w)$ as the {\bf gap-ratio of $K$ to $w$}. 
It is easy to know that for all $K$, we have $1\leq Gr(K,w)<\infty$ as $\dist(w,K)\neq0$. 
Particularly, $Gr(K,w)=1$ if $K$ is a single point and $Gr(K,w)=1$ if $K$ is a circle or an arc cetered at $w$.
\begin{figure}[h!]
	\centering
	{\includegraphics[scale=0.8]{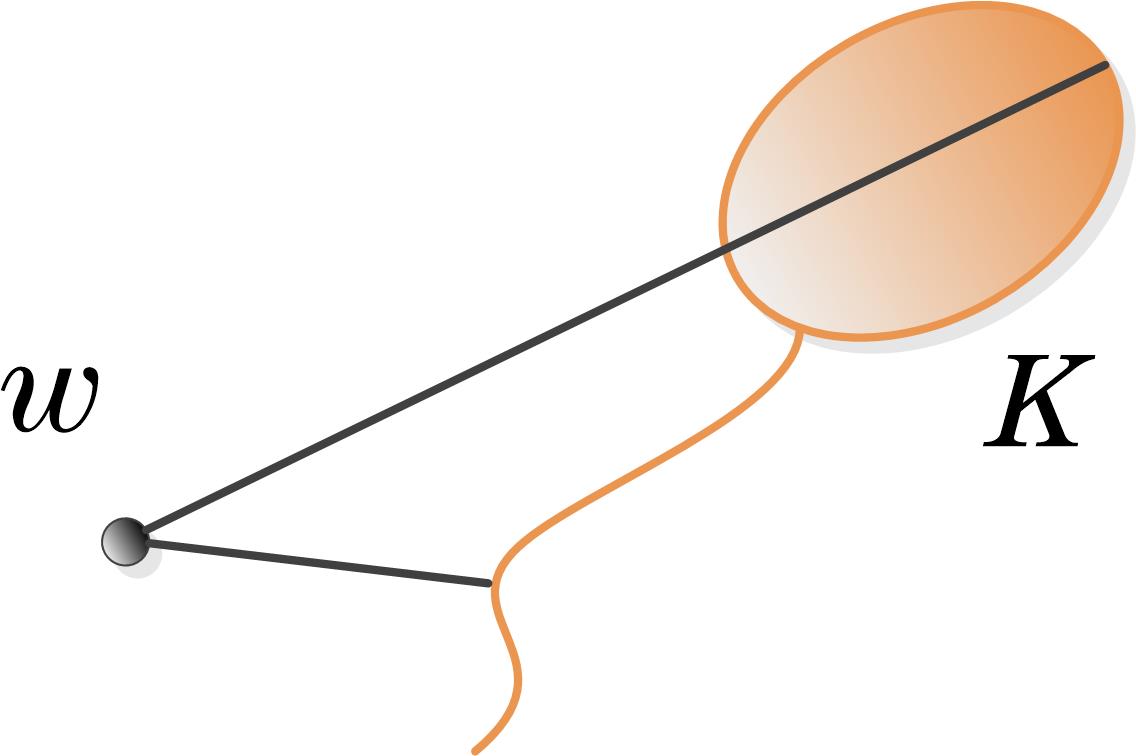}}
	\caption{}\label{gapratio}
\end{figure}

\begin{defn}
	Let $\Omega\subset\hat{\mathbb{C}}$ be a multi-connected domain with $\infty\in\Omega$. 
	For any $a,b\in\mathcal{C}(\Omega)$, the {\bf gap-ratio of $\ulcorner a\urcorner$ to $\ulcorner b\urcorner$} is
	\begin{equation*}
		Gr(\ulcorner a\urcorner,\ulcorner b\urcorner)=\displaystyle\sup_{w\in\ulcorner b\urcorner}\left(\frac{\displaystyle\sup_{z\in\ulcorner a\urcorner}\mid z-w\mid}{\displaystyle\inf_{z\in\ulcorner a\urcorner}\mid z-w\mid}\right).
	\end{equation*}
\end{defn}
Particularly, if we set $\ulcorner a\urcorner$ and $\ulcorner b\urcorner$ to be concentric circles, then 
$$Gr(\ulcorner a\urcorner,\ulcorner b\urcorner)=Gr(\ulcorner b\urcorner,\ulcorner a\urcorner)=1$$ 
for arbitrarily small $\dist(\ulcorner a\urcorner,\ulcorner b\urcorner)$.
\begin{defn}\label{gap-rho}
	Let $\Omega\subset\hat{\mathbb{C}}$ be a multi-connected domain with $\infty\in\Omega$. 
	We call $\Omega$ has bounded gap-ratio if for each $b\in\mathcal{C}(\Omega)$, 
	there is a $\delta>0$ such that for all $a\in\mathcal{C}(\Omega)$ satisfying
	$$\dist(\ulcorner a\urcorner,\ulcorner b\urcorner)=\displaystyle\inf_{z\in\ulcorner a\urcorner,w\in\ulcorner b\urcorner}\dist(z,w)<\delta,$$ 
	there must be
	\begin{equation}\label{gap}
		Gr(\ulcorner a\urcorner,\ulcorner b\urcorner)\leq\rho,
	\end{equation}
	where $\rho\geq1$ is a constant.
\end{defn}
We note that there is a subtle difference between the conditions of bounded gap-ratio and uniformly relative separation for nondegenerate domains. 
Let $\ulcorner b\urcorner$ be fixed and $\dist(\ulcorner a\urcorner,\ulcorner b\urcorner)$ tend to zero, 
the uniformly relative separation claims that the diameter of $\ulcorner a\urcorner$($\diam(\ulcorner a\urcorner)$) tends to zero.  
However, the condition (\ref{gap}) allows big $\diam(\ulcorner a\urcorner)$. For instance, we take $\ulcorner b\urcorner$ to be a disk, $\ulcorner a\urcorner$ to be an anulus sharing the same center of $\ulcorner b\urcorner$, then $\diam(\ulcorner a\urcorner)$ has a lower bound as $\dist(\ulcorner a\urcorner,\ulcorner b\urcorner)$ tends to zero. 
\begin{thm}\label{gapratio-well}
	Let $\Omega\subset\hat{\mathbb{C}}$ be a nondegenerate domain with bounded gap-ratio and $\infty\in\Omega$. Then $\mathcal{C}(\Omega)$ is well-distributed.
\end{thm}
\begin{proof}
	Let $\Omega_0$ be an arbitary domain that combined $\Omega$ with some components of $\hat{\mathbb{C}}-\Omega$. 
	For each $b\in\mathcal{C}(\Omega)$, recall that $N(\Omega_0,b)$ is the connected component of $\mathcal{E}(\Omega_0)-\beta(b)$ containing $b$. Take $w\in\ulcorner b\urcorner$ and $R_0>0$, 
	let $C(w,R_0)$ be the circle of radius $R_0$ cetered at $w$ such that $C(w,R_0)\subset\ulcorner N(\Omega_0,b)\urcorner$.
	Since $\beta(b)\in\Omega-\{\infty\}$, the complementary components intersecting $C(w,R_0)$ will not intersect $\beta(b)$. Hence they must be contained in $\ulcorner N(\Omega_0,b)\urcorner$. 
	Let $\mathcal{B}_{R_0}\subset\mathcal{C}_{nt}(\Omega)$ be the components that intersect both $C(w,R_0)$ and $C(w,R_0/2)$.
	Clearly, for each $a\in\mathcal{B}_{R_0}$, we have $\diam(\ulcorner a\urcorner)\geq R_0/2$. 
	This implies that $area(\ulcorner a\urcorner)\geq\kappa R^2_0/4$. 
	However, the area of $\ulcorner N(\Omega_0,b)\urcorner$ is limited, 
	we can conclude that $\mathcal{B}_{R_0}$ is a finite set. 
	Generally, we denote $\mathcal{B}_{R_j}(j=1,2,\ldots)$ as the collection of complementary components of $\Omega$ that intersect both $C(w,R_j)$ and $C(w,R_j/2)$, where $R_{j}=\min_{a\in\mathcal{B}_{R_{j-1}}}\dist(\ulcorner a\urcorner,w).$ The setting of $R_j$ ensures that there is no complementary component intersecting both $C(w,R_{j-1})$ and $C(w,R_j)$.
	
	Now we consider a subset of $\Gamma_{\Omega_0}(q,b)$, where $q\in\mathcal{E}(\Omega_0)$. Suppose that $\ulcorner q\urcorner$ is inside the circle $C(w,R_{2M+1})$ for some positive integer $M$. We take $r\in(R_{2M+1},R_1)$ and let $C(w,r)$ be the circle of radius $r$ centered at $w$. Moreover, we let $C'(w,r)=C(w,r)-\ulcorner b\urcorner$, then $C'(w,r)\cup\ulcorner b\urcorner$ separates $\ulcorner q\urcorner$ from $\infty$. Set $\gamma_r=\pi_{\Omega_0}(C'(w,r))$, we have $\{\gamma_r:r\in(R_M,R_1)\}\subset\Gamma_{\Omega_0}(q,b)$. For any extended metric $m$ on $\Omega_0$ such that $l=L_m(\Gamma_{\Omega_0}(q,b))>0$, we infer from the definition of $m-$length that
	\begin{equation}\label{l}
		l\leq L_m(\gamma_r)\leq\int_{C(w,r)\cap\Omega_0}m(z)\mid\mathrm{d}z\mid
		+\sum_{a\in\mathcal{C}(\Omega_0)\setminus\{b\}}\chi_a(C(w,r))m(a),
	\end{equation}
	where $\chi_a$ is defined as $\chi_a(A)=1$ if $\ulcorner a\urcorner\cap A\neq\emptyset$ and $\chi_a(A)=0$ if $\ulcorner a\urcorner\cap A=\emptyset$.
	Let $A_j$ be the annulus bounded by $C(w,R_j)$ and $C(w,R_{j}/2)$. In order to prevent the double counting of complementary components in the subsequent deduction, we integrate (\ref{l}) from $r=R_{2j}/2$ to $r=R_{2j}$ for $j=1,2,\cdots$. 
	Then we have
	\begin{equation*}
		\begin{aligned}
			l(R_{2j}-R_{2j}/2)&\leq\int^{R_{2j}}_{\frac{R_{2j}}{2}}\int_{C(w,r)\cap \Omega_0}m(z)|\mathrm{d}z|\mathrm{d}r+\sum_{a\in\mathcal{C}(\Omega_0)\setminus\{b\}}m(a)\int^{R_{2j}}_{\frac{R_{2j}}{2}}\chi_{a}(C(w,r))\mathrm{d}r\\
			&\leq\int_{A_{2j}\cap\Omega_0}m\mathrm{d}x\mathrm{d}y+\sum_{a\in\mathcal{C}(\Omega_0)\setminus\{b\}}m(a)\chi_{a}(A_{2j})\diam(\ulcorner a\urcorner)\\
			&\leq\int_{A_{2j}\cap\Omega_0}m\mathrm{d}x\mathrm{d}y+\sum_{a\in\mathcal{C}(\Omega_0)\setminus\{b\}}m(a)\chi_{a}(A_{2j})\sqrt{\kappa^{-1}area(\ulcorner a\urcorner)}.
		\end{aligned}
	\end{equation*}
	By Cauchy's inequality, 
	\begin{equation*}
		\begin{aligned}
			l^2(R_{2j}-R_{2j}/2)^2\leq&\left(\int_{A_{2j}\cap \Omega_0}m^2\mathrm{d}x\mathrm{d}y+\sum_{a\in\mathcal{C}(\Omega_0)\setminus\{b\}}m(a)^2\chi_{a}(A_{2j})\right)\\
			&\cdot\left(\int_{A_{2j}\cap \Omega_0}\mathrm{d}x\mathrm{d}y+\sum_{a\in\mathcal{C}(\Omega_0)\setminus\{b\}}\chi_{a}(A_{2j})\kappa^{-1}area(\ulcorner a\urcorner)\right).
		\end{aligned}
	\end{equation*}
	Since the complementary components that intersect $A_{2j}$ must be contained in $B(w,R_{2j-1})$, the sum of their area is controlled by the area of $B(w,R_{2j-1})$. Then we have
	\begin{equation*}
		\begin{aligned}
			\frac{l^2R^2_{2j}}{4}\leq&\left(area(B(w,R_{2j}))+\kappa^{-1}area(B(w,R_{2j-1}))\right)\\
			&\cdot\left(\int_{A_{2j}\cap \Omega_0}m^2\mathrm{d}x\mathrm{d}y+\sum_{a\in\mathcal{C}(\Omega_0)\setminus\{b\}}m(a)^2\chi_{a}(A_{2j})\right)\\
			\leq&\left(\pi R^2_{2j}+\kappa^{-1}\pi R^2_{2j-1}\right)\left(\int_{A_{2j}\cap \Omega_0}m^2\mathrm{d}x\mathrm{d}y+\sum_{a\in\mathcal{C}(\Omega_0)\setminus\{b\}}m(a)^2\chi_{a}(A_{2j})\right).
		\end{aligned}
	\end{equation*}
	This implies
	\begin{equation*}
		\begin{aligned}
			l^2\leq&\frac{\pi R^2_{2j}+\kappa^{-1}\pi R^2_{2j-1}}{R^2_{2j}/4}\left(\int_{A_{2j}\cap \Omega_0}m^2\mathrm{d}x\mathrm{d}y
			+\sum_{a\in\mathcal{C}(\Omega_0)\setminus\{b\}}m(a)^2\chi_{a}(A_{2j})\right)\\
			\leq&4\pi(1+\kappa^{-1})\frac{R^2_{2j-1}}{R^2_{2j}}\left(\int_{A_{2j}\cap \Omega_0}m^2\mathrm{d}x\mathrm{d}y
			+\sum_{a\in\mathcal{C}(\Omega_0)\setminus\{b\}}m(a)^2\chi_{a}(A_{2j})\right).
		\end{aligned}
	\end{equation*}
	Recall that $$R_{j+1}=\min_{a\in\mathcal{B}_{R_j}}dist(\ulcorner a\urcorner,w),\ \ \ j\in\mathbb{N}.$$
	There is at least one complementary component $a^*\in\mathcal{B}_{R_j}$ intersecting both $C(w,R_j)$ and $C(w,R_{j+1})$, but it does not contain any inner point of the disk bounded by $C(w,R_{j+1})$, see figure \ref{Rj-1Rj}. 
	\begin{figure}[h!]
		\centering
		{\includegraphics[scale=0.7]{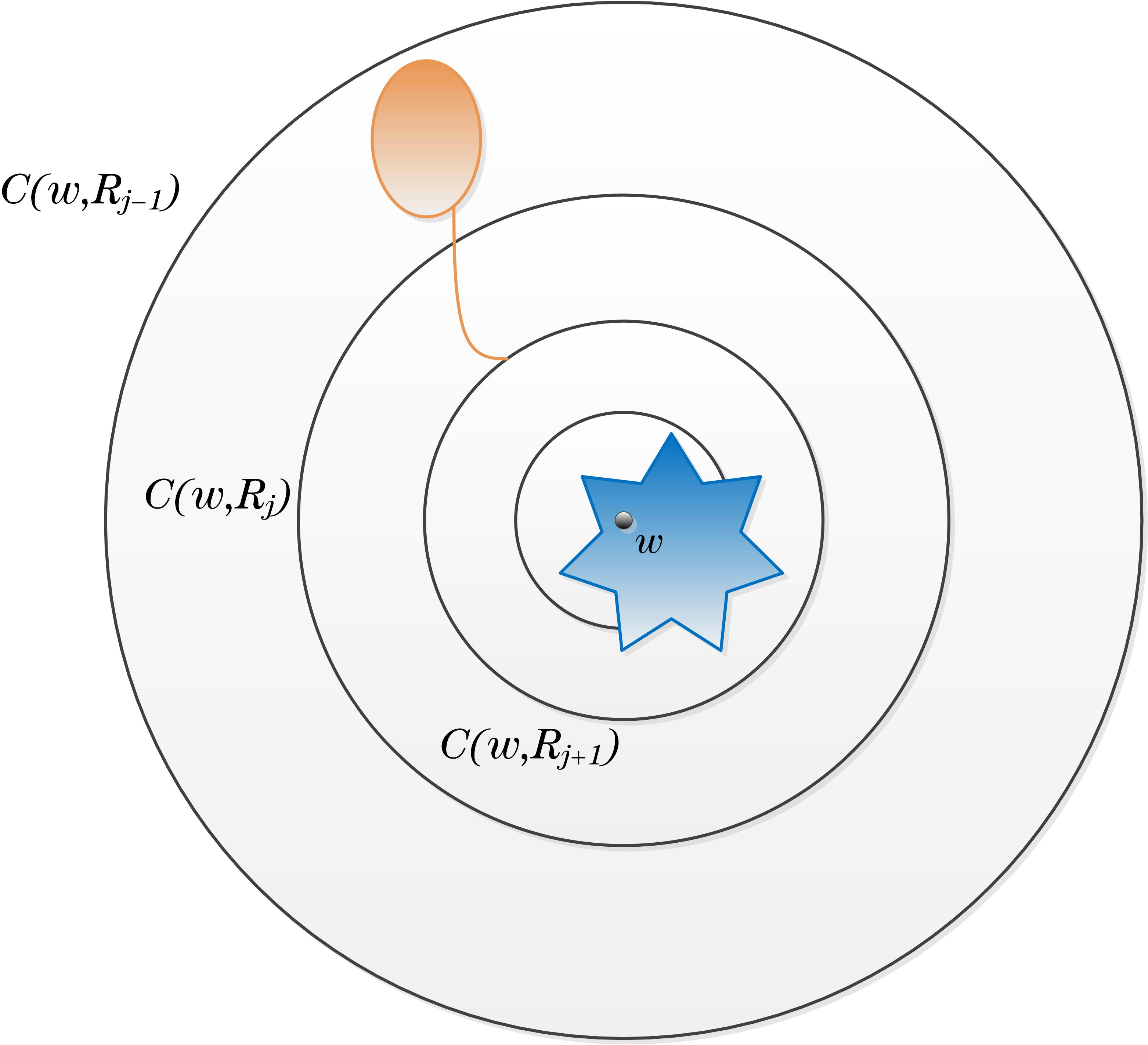}}
		\caption{}\label{Rj-1Rj}
	\end{figure}
	
	Since $\Omega$ is a domain with bounded gap-ratio, we take the first $R_0$ sufficiently small so that the gap-ratios of all complementary components inside $B(w,R_0)$ to $b$ are uniformly bounded. Hence we have
	\begin{equation*}
		\begin{aligned}
			\frac{R_{j}}{R_{j+1}}&\leq\frac{\displaystyle\sup_{z\in\ulcorner a^*\urcorner}\mid z-w\mid}{\displaystyle\inf_{z\in\ulcorner a^*\urcorner}\mid z-w\mid}
			\leq\displaystyle\max_{a\in\mathcal{B}_{R_j}}\displaystyle\sup_{w\in\ulcorner b\urcorner}\left(\frac{\displaystyle\sup_{z\in\ulcorner a\urcorner}\mid z-w\mid}{\displaystyle\inf_{z\in\ulcorner a\urcorner}\mid z-w\mid}\right)\\
			&\leq\displaystyle\max_{a\in\mathcal{B}_{R_j}}Gr(\ulcorner a\urcorner,\ulcorner b\urcorner)\leq\rho,
		\end{aligned}
	\end{equation*}
    where $\rho\geq1$ is a constant. 
	It then follows that
	\begin{equation}\label{l^2}
		l^2\leq4\pi\rho^2(1+\kappa^{-1})\left(\int_{A_{2j}\cap \Omega_0}m^2\mathrm{d}x\mathrm{d}y
		+\sum_{a\in\mathcal{C}(\Omega_0)-\{b\}}m(a)^2\chi_{a}(A_{2j})\right).
	\end{equation}
	
	We indicate that the complementary components intersecting $A_{2j}$ must be contained in $B(w,R_{2j-1})\setminus B(w,R_{2j+1})$, then for each $a\in\mathcal{C}(\Omega_0)$, there is at most one $j\in\{1,2,\cdots,M\}$ such that $\ulcorner a\urcorner\cap A_{2j}\neq\emptyset$. We add (\ref{l^2}) from $j=1$ to $j=M$, to get
	\begin{equation*}
		\begin{aligned}
			\sum^M_{j=1}l^2\leq&4\pi\rho^2(1+\kappa^{-1})
			\sum^M_{j=1}\left(\int_{A_{2j}\cap\Omega_0}m^2\mathrm{d}x\mathrm{d}y
			+\sum_{a\in\mathcal{C}(\Omega_0)\setminus\{b\}}m(a)^2\chi_{a}(A_{2j})\right)\\
			\leq&4\pi\rho^2(1+\kappa^{-1})\left(\int_{B(w,R_2)\cap\Omega_0}m^2\mathrm{d}x\mathrm{d}y
			+\sum_{a\in\mathcal{C}(\Omega_0)\setminus\{b\}}m(a)^2\chi_{a}(B(w,R_1))\right)\\
			\leq&4\pi\rho^2(1+\kappa^{-1})A(m).
		\end{aligned}
	\end{equation*}
	This implies $Ml^2\leq4\pi\rho^2(1+\kappa^{-1})A(m)$.
	Hence
	\begin{equation}\label{EL}
		EL(\Gamma)\leq4\pi\rho^2(1+\kappa^{-1})M^{-1}.
	\end{equation}

	Given $\varepsilon>0$, we can choose sufficiently large $M$ so that the right hand of (\ref{EL}) is less than $\varepsilon$. Then we have for each $w\in\ulcorner b\urcorner$, there is a $R_{2M+1}>0$ such that $EL(\Gamma)<\varepsilon$ when $d(\ulcorner q\urcorner,w)<R_{2M+1}$. According to the compactness of $\ulcorner b\urcorner$, the proof is completed.
\end{proof}

At the end of this section, we have derived from M\"obius invariance of fatness(\cite{Schramm}) and M\"obius invariance of crossratio that:
\begin{thm}[M\"obius invariance of gap-ratio]\label{mobius}
	Let $\Omega\subset\hat{\mathbb{C}}$ be a nondegenerate domain with bounded gap-ratio, and let $F$ be a M\"obius transformation. Then $F(\Omega)$ is also a nondegenerate domain with bounded gap-ratio.
\end{thm}

\section{The extended Carath\'eodory kernel convergence theorem}\label{sec4}
The Carath\'eodory kernel convergence theorem(\cite{Goluzin}) tells us what the image $f(\Omega)$ is. However, it does not give much informatin about $\ulcorner f(b)\urcorner$ for any $b\in\mathcal{C}(\Omega)$. This section is devoted to establish the extended Carath\'eodory kernel convergence theorem for nondegenerate domains.
\begin{thm}\label{ckc}
	Let $\Omega\subset\hat{\mathbb{C}}$ be a $\kappa-$nondegenerate domain with bounded gap-ratio. Given a sequence of conformal maps
	$\{f_n:\Omega\rightarrow\hat{\mathbb{C}},n=1,2,\cdots\}$ with the limit $f$. Suppose that for each $n=1,2,\cdots$, there is a domain $\Omega_n$ containing $\Omega$ such that
	\begin{itemize}
		\item $\Omega_n$ is a union of $\Omega$ and a collection of connected components of $\hat{\mathbb{C}}-\Omega$;
		\item $\mathcal{C}(\Omega_n)$ is at most countable;
		\item $f_n$ extends to a conformal mapping $\hat{f}_n:\Omega_n\rightarrow\hat{\mathbb{C}}$, and each $\hat{f}_n(\Omega_n)$ is a $\kappa-$nondegenerate domain.
	\end{itemize}
	Let $b\in\mathcal{C}(\Omega)$. Then $\ulcorner f(b)\urcorner$ is the complement of the connected component of $\hat{\mathbb{C}}-B_0^*$ that contains $f(\Omega)$, where $B_0^*$ is any Hausdorff limit of a subsequence of $\ulcorner f_n(b)\urcorner$. Moreover, we have $\ulcorner f(b)\urcorner$ is a singleton if $\ulcorner b\urcorner$ is.
\end{thm}
\begin{proof}
	We first normalize the sequence $f_n$ by requiring $\infty\in\Omega$ and $f_n(\infty)=\infty$ due to theorem \ref{mobius}. Without loss of generality, we assume that $B_0^*$ is the Hausdorff limit of the sequence $\ulcorner f_n(b)\urcorner$. Let $B^*$ be the complement of the connected component of $\hat{\mathbb{C}}-B_0^*$ that contains $\infty$. Our goal is to show that $B^*=\ulcorner f(b)\urcorner$. According to the Carath\'eodory kernel convergence theorem, we have $B^*\subset\ulcorner f(b)\urcorner$. Hence we only need to show that $\ulcorner f(b)\urcorner\subset B^*$.
	
	\vskip0.15in
	Striving for a contradiction, we assume $\ulcorner f(b)\urcorner-B^*\neq\emptyset$. Let $p^*\in\ulcorner f(b)\urcorner-B^*$ as shown in figure \ref{eta}. Set $\Omega^*_n=\hat{f}(\Omega_n)$,  $q^*_n=\pi_{\Omega^*_n}(p^*)$ and $q_n=\hat{f}_n^{-1}(q_n^*)$. (It is possible that $q_n\in\Omega_n$ or $q_n\in\mathcal{C}(\Omega_n)$.) Since $p^*\in\ulcorner f(b)\urcorner$, it follows that the distance in $\mathbb{C}$ between $\ulcorner q_n\urcorner$ and $\ulcorner b\urcorner$ tends to zero as $n\rightarrow\infty$. (Otherwise, we consider a Jordan curve in $\Omega$ which separates $\ulcorner b\urcorner$ from $\ulcorner q_n\urcorner$ for infinite many $n$, then the image of this curve under $f$ will separate $\ulcorner f(b)\urcorner$ from $p^*$.)
	\begin{figure}[h!]
		\centering
		{\includegraphics[scale=1]{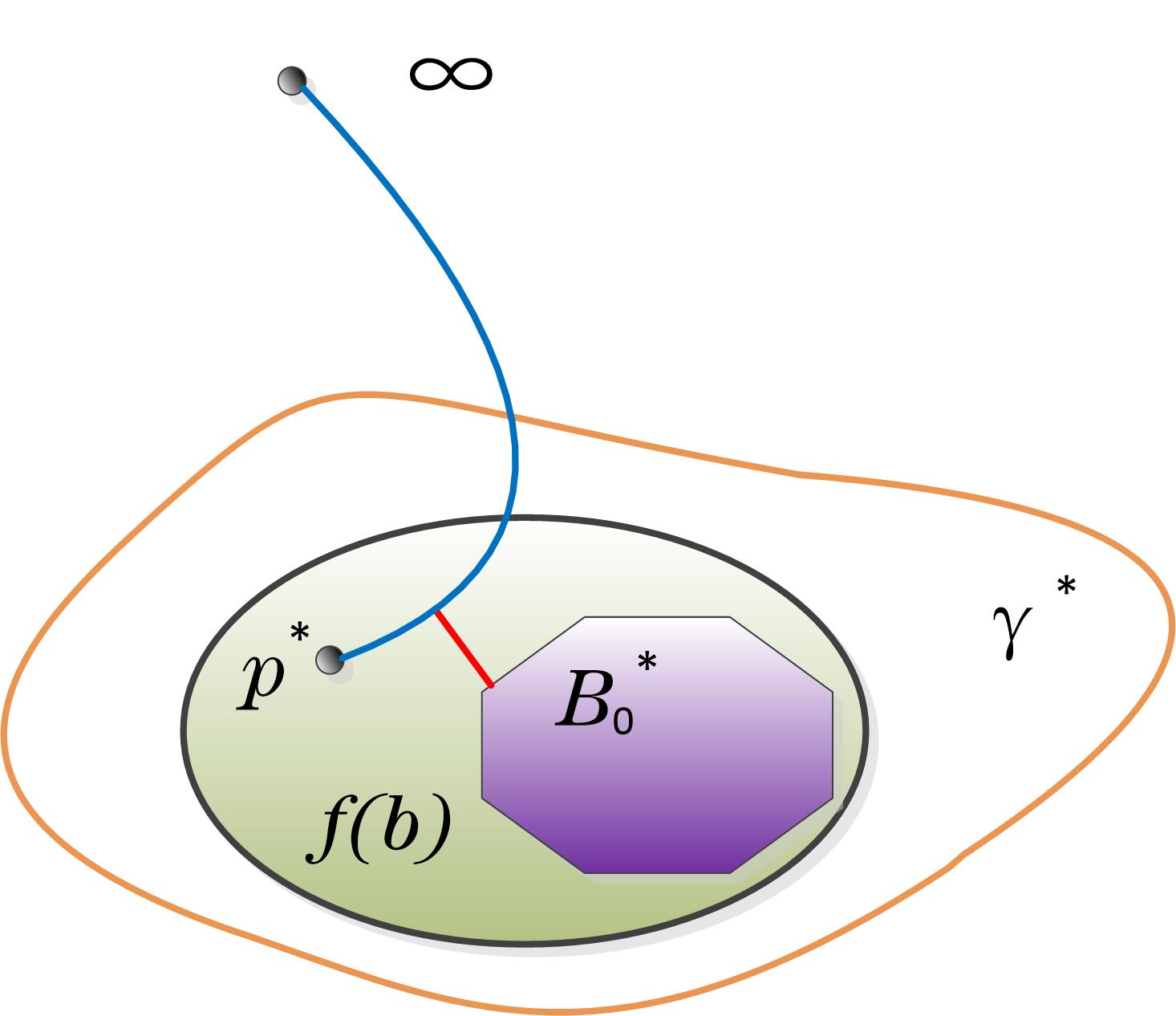}}
		\caption{}\label{eta}
	\end{figure}
	
	\vskip0.15in
	{\bf Part I: Prove that $EL(\Gamma_n)\rightarrow0$ as $n\rightarrow\infty$.}\par
	Recall that $\beta(b)\subset\Omega-\{\infty\}$ is a Jordan curve that 
	separates $b$ from $\infty$ in $\mathcal{E}(\Omega)$, 
	and $N(\Omega_n,b)$ is the connencted component of $\mathcal{E}(\Omega_n)\setminus\beta(b)$ that contains $b$. 
	Similarly, we let $N(\Omega^*_n,b)$ be the connencted component of $\mathcal{E}(\Omega^*_n)\setminus f_n(\beta(b))$ that contains $f_n(b)$.
	Let $\Gamma_n=\Gamma_{\Omega_n}(q,b)$ be the collection of all Jordan curves $\gamma\subset N(\Omega_n,b)\setminus\{b,q_n\}$ satisfies either of the following conditions:
	\begin{itemize}
		\item $\gamma$ separates $\{b,q_n\}$ from $\infty$;
		\item $\bar{\gamma}=\gamma\cup\{b\}$ and $\bar{\gamma}$ separates $q_n$ from $\infty$.
	\end{itemize}	
	
	Since $d(\ulcorner q_n\urcorner,\ulcorner b\urcorner)$ tends to zero as $n\rightarrow\infty$, 
	it follows from theorem \ref{gapratio-well} and definition \ref{well-distributed} that 
	$EL(\Gamma_n)$ tends to zero as $n\rightarrow\infty$. 
	
	\vskip0.15in
	Let $\Gamma^*_n=\{\hat{f}_n(\gamma):\gamma\in\Gamma_n\}$.
	
	\vskip0.15in
	{\bf Part II: Fix an extended metric to get a positive lower bound of $EL(\Gamma^*_n)$.}\par
	It follows from the normalization $f_n(\infty)=\infty$ that 
	there is some radius $R$ so that $f_{n}(\beta(b))\subset B(0,R)$ for all $n$. 
	We now define for all $z\in\Omega^*_n$:
	\begin{equation*}
		m^*_n(z)=
		\begin{cases}
			1, &z\in\Omega^*_n\cap B(0,R);\\
			0, &z\in\Omega^*_n-B(0,R),
		\end{cases}
	\end{equation*}
	and for $a^*\in\mathcal{C}(\Omega^*_n)$:
	\begin{equation*}
		m^*_n(a^*)=
		\begin{cases}
			\diam(\ulcorner a^*\urcorner), &\ulcorner a^*\urcorner\subset B(0,R);\\
			0, &\ulcorner a^*\urcorner\subset\hat{\mathbb{C}}-B(0,R).
		\end{cases}
	\end{equation*}
	It is clear that $m^*_n$ is an extended metric on $\mathcal{E}(\Omega^*_n)$. 
	In the following we will show that the ratio $L_{m^*_n}(\gamma^*)/A(m^*_n)$ 
	is strictly greater than zero under the metric $m^*_n$ for any $n$, 
	so that the transboundary extremal length $EL(\Gamma^*_n)$ is bounded away from zero as $n\rightarrow\infty$.
	
	\vskip0.15in
	Since $p^*\notin B^*$, there is a curve $\eta\subset\hat{\mathbb{C}}-B^*_0$ 
	that connects $p^*$ and $\infty$ as shown in figure \ref{eta}. 
	Let $\delta^*>0$ be a number that smaller than the distance from $\eta$ to $B^*_0$. 
	It is obvious that for sufficiently large $n$ 
	we have $\diam(\ulcorner\gamma^*\urcorner)>\delta^*$ for all $\gamma^*\in\Gamma^*_n$. 	
	
	Let $z_1,z_2\in\ulcorner\gamma^*\urcorner$, and for any $z\in\mathbb{C}$, set $\phi(z)=\mid z-z_1\mid$. 
	It is easy to know that $\phi(z)$ is a non-negative real valued function 
	and $\phi(\ulcorner\gamma^*\urcorner)$ covers the interval $[\ 0,\mid z_1-z_2\mid\ ]$.
	Parametering $\gamma^*$ as $\gamma^*:I\rightarrow\mathcal{E}(\Omega^*_n)$, where $I$ is an interval. 
	It is clear that $\gamma^{*-1}(\Omega^*_n)$ is a collection of connected components of $I(\gamma^{*-1}(\Omega^*_n))$. 
	We denote the collection of connected components as $\mathcal{J}$. 
	Then for each $J\in\mathcal{J}$, we have $\diam(\phi(\gamma^*(J)))\leq L_{m^*_n}(\gamma^*(J)).$
	
	On the other hand, for each $a^*\in\gamma^*\cap\mathcal{C}(\Omega^*_n)$, 
	it follows that
	$$\diam(\phi(\ulcorner\gamma^*(a^*)\urcorner))\leq \diam(\ulcorner\gamma^*(a^*)\urcorner)=m^*_n(a^*).$$ 
	Since the set of such $J$ and $a^*$ is countable, 
	and the interval $[\ 0,\mid z_1-z_2\mid\ ]$ is covered by the sets
	$$\{\phi(\gamma^*(J)):J\in\mathcal{J}\}\cup\{\phi(\ulcorner\gamma^*(a^*)\urcorner):a^*\in\gamma^*\cap\mathcal{C}(\Omega^*_n)\},$$
	this implies
	$$\mid z_1-z_2\mid\leq\sum_{J\in\mathcal{J}}L_{m^*_n}(\gamma^*(J))+\sum_{a^*\in\gamma^*\cap\mathcal{C}(\Omega_n)}m^*_n(a^*)=L_{m^*_n}(\gamma^*).$$
	Since $z_1,z_2$ are chosen arbitrarily, we conclude that
	\begin{equation}\label{lm-diam}
		L_{m^*_n}(\gamma^*)\geq \diam(\ulcorner\gamma^*\urcorner)>\delta^*.
	\end{equation}
	
	\vskip0.15in
	Moreover, we note that $\hat{f}_n(\Omega_n)$ is $\kappa-$nondegenerate. 
	Then for each non-trival component $a^*\in\mathcal{C}(\Omega^*_n)$, we have
	\begin{equation}
		\kappa(\ulcorner a^*\urcorner)=\frac{area(\ulcorner a^*\urcorner)}{\diam(\ulcorner a^*\urcorner)^2}\geq\kappa.
	\end{equation}
    Thus $m^*_n(a^*)^2\leq\kappa^{-1}area(\ulcorner a^*\urcorner)$ for $\ulcorner a^*\urcorner\subset B(0,R)$ and $m^*_n(a^*)^2=0$ for $\ulcorner a^*\urcorner\subset\hat{\mathbb{C}}-B(0,R)$.
	Moreover, we have
	\begin{equation}\label{area}
		A(m^*_n)\leq area(B(0,R)\cap\Omega^*_n)+\kappa^{-1}area(B(0,R))
		\leq(1+\kappa^{-1})\pi R^2.
	\end{equation}

    \vskip0.15in
	It then follows from (\ref{lm-diam}) and (\ref{area}) that
	\begin{equation*}
		EL(\Gamma^*_n)\geq\frac{(\delta^*)^2}{(1+\kappa^{-1})\pi R^2},
	\end{equation*}
	which implies that $EL(\Gamma^*_n)$ has a positive lower bound.
	
	\vskip0.15in
	According to the conformal invariance of transboundary extremal length, 
	we infer from the result in Part I that $\lim\limits_{n\rightarrow\infty}EL(\Gamma^*_n)=0$. 
	This contradiction establishes that $\ulcorner f(b)\urcorner=B^*$.
	
	\vskip0.15in
	{\bf Part III: Show that $\ulcorner f(b)\urcorner$ is a singleton if $\ulcorner b\urcorner$ is.}\par	
	We first suppose that $\ulcorner b\urcorner$ is a singleton and $n$ is a fixed positive integer. 
	Let $\Gamma_n$ be the collection of all Jordan curves in $N(\Omega_n,b)$ that separates $b$ from $\infty$
	and $\Gamma^*_n=\{\hat{f}_n(\gamma):\gamma\in\Gamma\}$. 
	We have $EL(\Gamma_n)=0$ from definition \ref{well-distributed}. 
	By the invariance of transboundary extremal length, we also have $EL(\Gamma^*_n)=0$. 
	This implies $\ulcorner\hat{f}_n(b)\urcorner$ is a single point. 
	Since this holds for every $n$ and $B^*_0$ is the Haussdorff limit of $\ulcorner f_n(b)\urcorner$, 
	$B^*_0$ must be a single point.
    Thus we have $\ulcorner f(b)\urcorner$ is a singleton.
\end{proof}

\section{Proof of the main theorem}\label{sec5}
Recall that a circle domain is a connected domain in $\hat{\mathbb{C}}$ such that 
every boundary component is either a cirlce or a point, 
hence it is $\pi/4-$nondegenerate.
\begin{proof}[Proof of theorem \ref{nondegenerate}]
	Let $\Omega\subset\hat{\mathbb{C}}$ be a nondegenerate domain with bounded gap-ratio. 
	Due to theorem \ref{mobius}, we assume $\infty\in\Omega$ for normalization.
	It is clear that the space $\mathcal{C}_{nt}(\Omega)$ must be countable. 
	Let $\mathcal{B}_1\subset\mathcal{B}_2\subset\cdots$ be a sequence of finite subsets of $\mathcal{C}_{nt}(\Omega)$ 
	such that $\cup^{\infty}_{n=1}\mathcal{B}_n=\mathcal{C}_{nt}(\Omega)$.
	For each $n=1,2,\ldots$, let $\Omega_n=\hat{\mathbb{C}}-\cup_{b\in\mathcal{B}_n}\ulcorner b\urcorner$. 
	By Koebe's finite connected uniformization theorem there is a conformal homeomorphism $F_n$ 
	such that $F_n(\infty)=\infty$ and $F_n(\Omega_n)$ is a circle domain. 
	Let $f_n$ be the limitation of $F_n$ on $\Omega$, 
	thus we obtain a normal family $\{f_n\}^\infty_{n=1}$. 
	Without loss of generality, we assume $f_n$ converges to $f$. 
	It is clear that $f$ is conformal on $\Omega$ and satisfies the normalization. 
	According to theorem \ref{ckc}, we have $\ulcorner f(b)\urcorner$ is a round disk 
	for each $b\in\bigcup^{\infty}_{n=1}\mathcal{B}_{n}$, 
	and $\ulcorner f(b)\urcorner$ is a single point 
	for each $b\in\mathcal{C}({\Omega})-\bigcup^{\infty}_{n=1}\mathcal{B}_{n}$. 
	This implies $f(\Omega)$ is a circle domain.
\end{proof}

\end{document}